\documentclass[12pt]{amsart}
\usepackage{geometry}                
\usepackage{graphicx}
\usepackage{amssymb}
\usepackage{epstopdf}
\theoremstyle{plain}
\newtheorem{thm}{Theorem}[section]
\newtheorem{lem}[thm]{Lemma}
\newtheorem{prop}[thm]{Proposition}
\newtheorem{cor}[thm]{Corollary}
\newtheorem{rmk}[thm]{Remark}
\newtheorem{ques}[thm]{Question}
\newcommand{\Cl}{\operatorname{Cliff}(C)}
\newcommand{\Cln}{\operatorname{Cliff}_n(C)}
\newcommand{\Clt}{\operatorname{Cliff}_2(C)}
\newcommand{\Clth}{\operatorname{Cliff}_3(C)}
\newcommand{\Clo}{\operatorname{Cliff}_1(C)}
\newcommand{\Clnn}{\operatorname{Cliff}_{n+1}(C)}
\newcommand{\Clr}{\operatorname{Cliff}_r(C)}
\title{On an example of Mukai}
\author{H. Lange, V. Mercat and P. E. Newstead}
\address{H. Lange\\Department Mathematik\\
              Universit\"at Erlangen-N\"urnberg\\
              Bismarckstra\ss e $1\frac{ 1}{2}$\\
              D-$91054$ Erlangen\\
              Germany}
              \email{lange@mi.uni-erlangen.de}
\address{V. Mercat\\5 rue Delouvain\\75019 Paris\\France}
              \email{vipani.mm@gmail.com}
\address{P.E. Newstead\\Department of Mathematical Sciences\\
              University of Liverpool\\
              Peach Street, Liverpool L69 7ZL, UK}
\email{newstead@liv.ac.uk}

\thanks{All authors are members of the research group VBAC (Vector Bundles on Algebraic Curves).}

\begin{document}
\begin{abstract}
In this note we use an example of Mukai to construct semistable bundles of rank 3 with 6 independent sections on a general curve of genus 9 or 11 with Clifford index strictly less than the Clifford index of the curve. The example also allows us to show the non-emptiness of some Brill-Noether loci with negative expected dimension.
\end{abstract}

\maketitle
\section{Introduction}\label{intro}
In \cite[Proposition 2]{mu1}, Mukai stated that, if $C$ is a non-pentagonal curve of genus $9$, then there exists a unique stable bundle $E$ of rank 3 and determinant $K_C$ with $h^0(E)=6$ (actually Mukai said ``quasi-stable'', which is what is now usually called ``polystable'', but since $\deg E=16$ is coprime to $3$, this implies that $E$ is stable). Computing the Clifford index $\gamma(E)$ as defined in \cite{cl} gives
$$\gamma(E)=\frac13(16-6)=\frac{10}3.$$
Since $C$ has Clifford index $\Cl=4$, this contradicts the conjecture of \cite{m} (see also \cite[Conjecture 9.3]{cl}). It shows further that the Brill-Noether locus $B(3,16,6)$, which has ``expected dimension'' $-11$, is non-empty. It also sheds light on the main result of \cite{cl1}, which implies that any semistable bundle of rank $3$ with $h^0=6$ on $C$ has degree $\ge15$; in fact such a bundle of degree $15$ cannot exist (see Proposition \ref{comm1} and Comment 1), so Mukai's bundle has the minimum possible degree for a semistable bundle of rank $3$ with $h^0=6$.

In fact, \cite{mu1} contains no proofs. The above result is proved in \cite{mu2}, except that the stability of $E$ is only indicated in a remark \cite[Remark 5.7(2)]{mu2} (the full proof may be found in the addendum \cite{mu3}). In this note, we give a complete proof, show that a similar result holds for genus $11$ and consider possible generalisations and extensions; in the main theorem (Theorem \ref{main}), we give general conditions under which the Clifford index $\Cln$ defined in \cite{cl} is strictly smaller than $\Cl$ either for $n=2$ or for $n=3$. The methods are for the most part those of Mukai. 

In a postscript, we comment on developments since this paper was completed.

We would like to thank the referee for some helpful comments.

\section{Background and preliminaries}\label{back}
Let $C$ be a smooth projective curve of genus $g \geq 4$ defined over an algebraically closed field of characteristic zero. We denote by $K_C$ the canonical line bundle on $C$.
In  \cite{cl}, the classical Clifford index $\Cl$ of $C$ was generalised to semistable bundles in two different ways, only one of which is needed in this paper. First we define, for any vector bundle $E$ of rank $n$ and degree $d$, 
$$
\gamma(E) := \frac{1}{n} \left(d - 2(h^0(E) -n)\right) = \mu(E) -2\frac{h^0(E)}{n} + 2,
$$
where $\mu(E)=\frac{d}n$.
Then 
the Clifford index $\Cln$ is defined by 
$$
\Cln := \min_{E} \left\{ \gamma(E) \;\left| 
\begin{array}{c} E \;\mbox{semistable of rank}\; n \\
h^0(E) \geq 2n,\; \mu(E) \leq g-1
\end{array} \right\} \right.
$$
(this invariant is denoted by $\gamma_n'$ in \cite{cl, ln, cl1, ln2, ln3}). We say that a bundle $E$ {\em contributes} 
to $\Cln$ if it is semistable of rank $n$ with $\mu(E) \leq g-1$ and $h^0(E)\ge 2n$ and that $E$ 
{\em computes} $\Cln$ if in addition $\gamma(E)=\Cln$. Note that $\Clo=\Cl$.

A conjecture was made in \cite{m} concerning the maximum value of $h^0(E)$ for $E$ a semistable bundle of any given rank and degree; the most important part of this conjecture can be stated as \\

\noindent{\bf Conjecture.} \cite[Conjecture 9.3]{cl} $\Cln=\Cl$.\\

The main purpose of this note is to give examples to show that this conjecture can fail.
For this purpose we use a bundle of rank $3$ on a  curve of genus $9$ constructed by Mukai (see \cite{mu1, mu2, mu3}). We show that the construction works also for genus $11$.

In the remainder of this section, we introduce some notation which we shall need. First we recall the {\em gonality sequence} 
$d_1,d_2,\ldots,d_r,\ldots$ of $C$ defined by 
$$
d_r := \min \{ \deg L \;|\; L \; \mbox{a line bundle on} \; C \; \mbox{with} \; h^0(L) \geq r +1\}.
$$
Note that a line bundle $L$ of degree $d_r$ with $h^0(L)\ge r+1$ in fact has $h^0(L)=r+1$ and is  generated by its sections, so we have an exact evaluation sequence
\begin{equation}\label{eili}
0\to E^*\to H^0(L)\otimes{\mathcal O}_C\to L\to0.
\end{equation}
The  bundle $E$ is often called the {\em dual span} of $L$.
Note that $\Cl$ is the minimum value of $d_r-2r$ taken over all $r$ for which $d_r\le g-1$. 
We shall say that $d_r$ {\em computes} $\Cl$ if $d_r\le g-1$ and $d_r-2r=\Cl$. The numbers $d_r$ satisfy the inequalities
\begin{equation}\label{dr}
d_r\le g+r-\left[\frac{g}{r+1}\right]
\end{equation}
with equality if $C$ is a {\em Petri curve}, that is a curve for which the multiplication map
$$H^0(L)\otimes H^0(L^*\otimes K_C)\longrightarrow H^0(K_C)$$
is injective for every line bundle $L$ on $C$. It is important to note that the general curve of any given genus is a Petri curve (see, for example, \cite{Laz}).

We need also to recall the definitions of the higher rank Brill-Noether loci (see \cite{gt} for a survey of the theory and \cite{bgn} for the notations we use here). Let $M(n,d)$ denote the moduli space of stable bundles of rank $n$ and degree $d$ on $C$. For any positive integer $k$, the {\em Brill-Noether locus} $B(n,d,k)$ is defined by
$$B(n,d,k)=\{E\in M(n,d)|h^0(E)\ge k\}.$$
For every irreducible component $Z$ of $B(n,d,k)$, we have
$$\dim Z\ge\beta(n,d,k):=n^2(g-1)-k(k-d+n(g-1)).$$
The number $\beta(n,d,k)$ is often called the {\em expected dimension} of $B(n,d,k)$.

Throughout the paper, $C$ will denote a smooth projective curve of genus $g$ and Clifford 
index $\Cl\ge3$ (hence also $g\ge7$ by \eqref{dr}) defined 
over an algebraically closed field of characteristic zero. For a vector bundle $G$ on $C$, the rank and 
degree of $G$ will be denoted by $r_G$ and $d_G$ respectively. 

The following lemma of Paranjape-Ramanan will be used on several occasions.
 \begin{lem}\label{lempr}
Let $E$ be a vector bundle on $C$ of rank $n\ge2$ with $h^0(E)=n+s$ ($s\ge1$) such that $E$ possesses 
no proper subbundle $F$ with $h^0(F)>r_F$. Then $h^0(\det E)\ge ns+1$ and so $d_E\ge  d_{ns}$.\end{lem}
\begin{proof}
This is a restatement of \cite[Lemma 3.9]{pr}.
\end{proof}

\section{The main theorem}\label{mainsec}
Let $C$ be a smooth curve with $\Cl\ge3$  and $n$ an integer, $n\ge3$. We begin by taking 2 line bundles $L_1$, $L_2$ on $C$ of degree $d_{n-1}$ with $h^0(L_i)=n$. Let $E_i$ denote the dual span of $L_i$ (see \eqref{eili}), so we have exact sequences
\begin{equation}\label{extra}
0\to E_i^*\to H^0(L_i)\otimes{\mathcal O}_C\to L_i\to0.
\end{equation}
\begin{lem}\label{dspan}
Suppose $\frac{d_p}{p}\ge\frac{d_{p+1}}{p+1}$ for all $p<n-1$ and $d_{n-1}\ne (n-1)d_1$. Then $E_i$ is semistable and $h^0(E_i)=n$ for $i=1,2$.
\end{lem}
\begin{proof} Dualising \eqref{extra}, we see at once that $h^0(E_i)\ge n$. The result now follows from \cite[Proposition 4.9(d) and Theorem 4.15(a)]{cl}.
\end{proof}
Now consider non-trivial extensions 
\begin{equation}\label{e1l2}
0\to E_1\to E\to L_2\to 0.
\end{equation}
Note that, if the hypotheses of Lemma \ref{dspan} hold, then $h^0(E_1)=h^0(L_2)=n$, so $h^0(E)\le2n$.
\begin{lem}\label{nontriv}
Suppose that $h^0(E_1)=n$. Then there exists a non-trivial extension \eqref{e1l2} with $h^0(E)=2n$ if and only if $h^0(E_2\otimes E_1)>n^2$.
\end{lem}
\begin{proof} Clearly $h^0(E)=2n$ if and only if all sections of $L_2$ lift to $E$. We consider the dual of the sequence \eqref{extra} for $i=2$ tensored by $E_1$. If $\alpha:L_2\to E_1$ is a non-zero homomorphism, then $H^0(L_2)\subset H^0(E_1)$. Since both spaces have dimension $n$, it follows that all sections of $E_1$ have the form $\alpha\circ s$ for some $s\in H^0(L_2)$, contradicting the fact that $E_1$ is generated. So $H^0(L_2^*\otimes E_1)=0$ and, taking cohomology, we obtain an exact sequence
\begin{equation}\label{long}
0\to H^0(L_2)^*\otimes H^0(E_1)\stackrel{\psi}{\to} H^0(E_2\otimes E_1)\to H^1(L_2^*\otimes E_1)\stackrel{\varphi}{\to} H^0(L_2)^*\otimes H^1(E_1).
\end{equation}
An extension \eqref{e1l2} has the property that all sections of $L_2$ lift if and only if its class is in $\ker\varphi$. So there exists such an extension with this property if and only if $\psi$ fails to be surjective. Since 
$H^0(L_2)^*\otimes H^0(E_1)$ has dimension $n^2$, the result follows.
\end{proof}
\begin{lem}\label{cliff}
Suppose $E$ is as in Lemma \ref{nontriv}. If $\gamma(E)<\Cl$, then $2d_{n-1}<nd_1$. Conversely, if $d_1$ computes $\Cl$ and $2d_{n-1}<nd_1$, then $\gamma(E)<\Cl$.
\end{lem}
\begin{proof}
We have $\gamma(E)=\frac1n(2d_{n-1}-2n)=\frac{2d_{n-1}}n-2$. Since $\Cl\le d_1-2$ with equality if $d_1$ computes $\Cl$, the result follows.
\end{proof}
\begin{lem}\label{n=3}
Suppose $C$ is a Petri curve of genus $g\ge7$ and $n\ge3$. Then $2d_{n-1}<nd_1$ except when $n=3, g=8, 10, 14$.
\end{lem} 
\begin{proof} This follows by direct computation from the formulae for $d_r$ (see \eqref{dr}).
\end{proof}

\begin{prop}\label{ss}
Suppose that $d_1$ computes $\Cl$, $3d_1\ge2d_2$ and $\Clt=\Cl$. Suppose further that $n=3$ and the extension \eqref{e1l2}
$$0\to E_1\to E\to L_2\to 0$$ 
is non-trivial with $h^0(E)=6$. Then $E$ is semistable.
\end{prop}
\begin{proof}
Suppose first that $F$ is a subbundle of $E$ of rank $2$ contradicting semistability. Then $d_{E/F}<\frac{2d_2}3$. Since $E$ is generated with $h^0(E^*)=0$, the same holds for $E/F$, so $E/F\not\simeq{\mathcal O}$ and hence $h^0(E/F)\ge2$ and $d_{E/F}\ge d_1$. This contradicts the hypothesis $3d_1\ge2d_2$.

Now suppose there is no subbundle of rank $2$ contradicting semistability, but that $L$ is a line subbundle with $d_L>\frac{2d_2}3$.  If $E/L$ is not stable, we can pull back a line subbundle of $E/L$ to get a subbundle $F$ of $E$ of rank $2$ with
$$d_F\ge\frac12d_{E/L}+d_L=\frac12(d_E+d_L)>\frac{4d_2}3.$$
This is a contradiction, so $E/L$ is stable. Since $L\not\subset E_1$ and \eqref{e1l2} is assumed non-trivial, we have $d_L<d_2$. So $h^0(L)\le2$ and $h^0(E/L)\ge4$. It follows that $E/L$ contributes to $\Clt$, so 
\begin{equation}\label{E/L}
d_{E/L}\ge2\Clt+4=2\Cl+4=2d_1.
\end{equation}
But $d_{E/L}<\frac{4d_2}3$, so this again contradicts the hypothesis $3d_1\ge2d_2$.
\end{proof}
\begin{thm}\label{main}
Suppose $C$ is a curve for which $d_1$ computes $\Cl$, $3d_1>2d_2$ and there exist $L_1$, $L_2$ of degree $d_2$ with $h^0(L_i)=3$ for which $h^0(E_2\otimes E_1)>9$. Then either $\Clt<\Cl$ or $\Clth<\Cl$.
\end{thm}
\begin{proof}
This follows from Lemmas \ref{dspan}, \ref{nontriv} and \ref{cliff} and Proposition \ref{ss}.
\end{proof}
\begin{rmk}\label{2<1}\begin{em}
Suppose that all the hypotheses of Proposition \ref{ss} hold except that $\Clt<\Cl$. Then $\Cl\ge5$ by \cite[Proposition 3.8]{cl} and hence $g\ge11$; moreover, by \cite[Theorem 5.2]{cl}, we have $\Clt\ge\frac{d_4}2-2$. It follows that the proof of the proposition is valid except that \eqref{E/L} must be replaced by
\begin {equation}\label{E/L2}
d_{E/L}\ge d_4.
\end{equation}
If $3d_4\ge4d_2$, this contradicts the assumption that $E$ is not semistable. 
\end{em}\end{rmk}
\begin{cor}\label{maincor}
Suppose $C$ is a curve for which $d_1$ computes $\Cl$, $3d_1>2d_2$,  $3d_4\ge4d_2$ and there exist $L_1$, $L_2$ of degree $d_2$ with $h^0(L_i)=3$ for which $h^0(E_2\otimes E_1)>9$. Then $\Clth<\Cl$.
\end{cor}
\begin{proof}
This follows from the proof of Theorem \ref{main} and Remark \ref{2<1}.
\end{proof} 
\begin{rmk}\label{d4}\begin{em}
The assumption $3d_4\ge4d_2$ holds for Petri curves of genus $12$, $13$, $14$, $15$, $18$, $19$ and $24$ (also for genus $\le10$, but in this case $\Clt=\Cl$, so the corollary does not lead to any improvement).
\end{em}\end{rmk} 

\section{Curves of genus $9$ and $11$}\label{911}
To find examples of curves for which $\Clth<\Cl$, it remains only to choose $C$ suitably and then show that there exist $L_1$, $L_2$ as in the statement of Theorem \ref{main} such that  $h^0(E_2\otimes E_1)>9$. 

For $g=9$, Mukai \cite[Proposition 1.2]{mu2} proves this by a  very special argument which works also for $g=11$.  This is based on a result of Mumford \cite{mm}; for completeness and in view of possible generalisations, we give a proof using only Mumford's result.
\begin{prop}\label{mukai}
With the notations of the previous section, suppose $n=3$ and let $C$ be a Petri curve of genus $9$ or $11$. Then there exist $E_1$, $E_2$ as above such that $h^0(E_2\otimes E_1)>9$.
\end{prop}
\begin{proof}
Note first (see \eqref{dr}) that $d_2=g+2-\left[\frac{g}3\right]=g-1$ in both cases. Take any line bundle $L_1$ of degree $d_2$ with $h^0(L_1)=3$ and put $L_2=L_1^*\otimes K_C$. Then $d_{L_2}=d_2$ and $h^0(L_2)=3$. Moreover
$$\det E_1\otimes \det E_2\simeq L_1\otimes L_2\simeq K_C.$$
The canonical homomorphism
$$(E_2\otimes E_1)\otimes(E_2\otimes E_1)\to\bigwedge^2E_2\otimes\bigwedge^2E_1\simeq K_C$$ 
defines a $K_C$-valued quadratic form $Q$ on $E_2\otimes E_1$. 
Now let $L$ be a line subbundle of $E_1$, put $M=E_1/L$ and consider the family of  extensions
$$0\to L\to F\to M\to0$$
with $L$ and $M$ fixed. Tensoring by $E_2$, we obtain a family
\begin{equation}\label{eqn:fam}
0\to E_2\otimes L\to E_2\otimes F\to E_2\otimes M\to 0.
\end{equation}
The quadratic form $Q$ extends to a quadratic form on the family \eqref{eqn:fam}. It follows from the theorem of  \cite{mm} (see \cite[p186, Application (5)]{mm}) that
$$h^0(E_2\otimes E_1)\equiv h^0(E_2\otimes L)+h^0(E_2\otimes M)\bmod2.$$
On the other hand, by Serre duality and Riemann-Roch,
$$h^0(E_2\otimes L)+h^0(E_2\otimes M)=h^0(E_2\otimes L)+h^1(E_2\otimes L)\equiv\deg E_2\bmod2.$$
Now $\deg E_2=g-1$ is even, so $h^0(E_2\otimes E_1)$ is even. The result now follows from \eqref{long}.
\end{proof}
\begin{rmk}\label{r1}
\begin{em}The only place where $C$ being Petri is used here is in the requirement that $d_2=g-1$. For genus $9$, this is true whenever $\Cl$ takes its maximum value $4$. In fact, we have in this case $d_2\ge \Cl+4=8$ and (by \eqref{dr}) $d_2\le g+2-\left[\frac{g}3\right]=8$. We have also $d_1=6$ (in the language of Mukai's papers, $C$ is ``non-pentagonal'') and hence $3d_1>2d_2$.
\end{em}\end{rmk}

\begin{thm}\label{g=9} Let $C$ be a curve of genus $9$ with Clifford index $\Cl=4$. Then \begin{itemize}
\item[(i)] $\Clr<\Cl$ whenever $r$ is divisible by $3$;
\item[(ii)] $3\le\Clth\le\frac{10}3$;
\item[(iii)] the expected dimension of the Brill-Noether locus $B(3,16,6)$ is $-11$, but this locus is non-empty.
\end{itemize}
\end{thm}
\begin{proof} 
(i) If $r$ is divisible by $3$, $\Clr\le \Clth$ by \cite[Lemma 2.2]{cl}, so it is sufficient to prove that $\Clth<\Cl$. When $\Cl=4$, we have $\Clt=\Cl$ by \cite[Proposition 3.8]{cl}. The result now follows from Theorem \ref{main}, Proposition \ref{mukai} and Remark \ref{r1}.

(ii) By Propositions \ref{mukai} and \ref{ss}, there exists a semistable bundle $E$ of rank $3$ and degree $2d_2=16$ with $h^0(E)=6$. Hence $\gamma(E)=\frac{10}3$, so $\Clth\le\frac{10}3$. On the other hand, by \cite[Theorem 4.1]{cl1}, $\Clth\ge3$.

(iii) We have $\beta(3,16,6)=9g-8-6(6-16+3g-3)=-11$. The bundle $E$ is semistable and therefore stable since $\gcd(3,16)=1$, so $B(3,16,6)\ne\emptyset$. 
\end{proof}

For $g=11$, there is more work to do. For a Petri curve of genus $11$, we have $\Cl=5$ and $d_4=13$. It follows from  \cite[Theorem 5.2]{cl} that
\begin{equation}\label{gamma2'}
\frac92\le\Clt\le5.
\end{equation} 
We wish to investigate the possibility that $\Clt=\frac92$. We begin with a lemma which generalises part of \cite[Theorem 5.2]{cl}.
\begin{lem}\label{gen}
Let $C$ be any smooth curve and $F$ a semistable bundle of rank $2$ and slope $\mu(F)\le g-1$ on $C$ with $h^0(F)=n+s$, $s>0$. Then
$$\gamma(F)\ge\min\left\{\Cl,\frac{d_{2s}}2-s\right\}.$$
\end{lem}
\begin{proof}
If $F$ has a line subbundle $L$ with $h^0(L)\ge2$, then, as in the first part of the proof of \cite[Theorem 5.2]{cl}, we have $\gamma(F)\ge\Cl$. Otherwise, by Lemma \ref{lempr}, $d_F\ge d_{2s}$, giving $\gamma(F)\ge\frac{d_{2s}}2-s$.
\end{proof}
\begin{prop}\label{g11<}
Let $C$ be a curve of genus $11$ with $\Cl=5$ and $\Clt<\Cl$. Then $\Clt$ is computed by one or more generated  stable bundles $F$ of rank $2$ and degree $13$ with $h^0(F)=4$ and by no other bundles.
\end{prop}
\begin{proof}
We have $\Clt=\frac92$ by \eqref{gamma2'}. Let $F$ be a bundle computing $\Clt$ with $h^0(F)=2+s$. If $s\ge3$, then, by Lemma \ref{gen}, 
$$ \gamma(F)\ge\min\left\{\Cl,\frac{d_{2s}}2-s\right\}\ge\min\left\{5,\frac{d_{6}+2s-6}2-s\right\}=\min\left\{5,\frac{d_{6}}2-3\right\}=5,$$
since $d_6=16$ by \eqref{dr} and \cite[Lemma 4.6]{cl}, a contradiction. So $s=2$, giving $h^0(F)=4$ and $d_F=13$. Since $F$ is semistable and $\gcd(2,13)=1$, it is in fact stable. If $F$ is not generated, then there exists a subsheaf $F'$ of degree $12$ with $h^0(F')=4$; moreover $F'$ is a semistable bundle. Then $\gamma(F')=4$, a contradiction.
\end{proof}
\begin{thm}\label{g=11}
Let $C$ be a Petri curve of genus $11$. Then 
\begin{itemize}
\item[(i)] $\Clr<\Cl$ whenever $r$ is divisible by $3$;
\item[(ii)] $\frac{11}3\le\Clth\le\frac{14}3$;
\item[(iii)] the expected dimension of the Brill-Noether locus $B(3,20,6)$ is $-5$, but this locus is non-empty.
\end{itemize}
\end{thm}
\begin{proof} 
(i) If $\Clt=\Cl$, this follows from Theorem \ref{main}, Lemma \ref{n=3} and Proposition \ref{mukai} together with \cite[Lemma 2.2]{cl}.

If $\Clt<\Cl$, Theorem \ref{main} does not apply and, since $d_2=10$ and $d_4=13$, neither does Corollary \ref{maincor}. However, by Proposition \ref{g11<}, there exists a generated stable bundle $F$ of rank $2$ and degree $13$ with $h^0(F)=4$. Any line subbundle of $F$ has degree $\le6$ and hence $h^0\le1$. So, by Lemma \ref{lempr}, we have $h^0(\det F)\ge5$, and in fact this is an equality since $\deg\det F=13=d_4$.

Now let $L=\det F$ and consider extensions
\begin{equation}\label{ext}
0\to F\to E\to K_C\otimes L^*\to0.
\end{equation}
It is clear that $\deg E=20$ and that, if \eqref{ext} does not split, then $E$ is stable. Moreover, since $h^0(L)=5$, it follows by Riemann-Roch that $h^0(K_C\otimes L^*)=2$. Thus $h^0(E)\le6$ with equality if and only if all sections of $K_C\otimes L^*$ lift to $E$. For this, we require that the canonical homomorphism
$$H^1(K_C^*\otimes L\otimes F)\to \mbox{Hom}(H^0(K_C\otimes L^*),H^1(F))$$
should fail to be injective.  Equivalently, the dual homomorphism
\begin{equation}\label{final}
H^0(K_C\otimes L^*)\otimes H^0(K_C\otimes F^*)\to H^0(K_C^2\otimes L^*\otimes F^*)
\end{equation}
should be non-surjective. We already know that $h^0(K_C\otimes L^*)=2$ and, by Riemann-Roch,
$$h^0(K_C\otimes F^*)=h^1(F)=h^0(F)-13+20=11.$$
So the dimension of the LHS of \eqref{final} is equal to $22$. The bundle $K_C^2\otimes L^*\otimes F^*$ is a stable bundle of rank $2$ and degree $41=4g-3$, so, by Riemann-Roch, the RHS of \eqref{final} has dimension $2g-1=21$. Moreover, by the base point free pencil trick, the kernel of \eqref{final} is isomorphic to
$$H^0(K_C^*\otimes L\otimes K_C\otimes F^*)=H^0(\det F\otimes F^*)=H^0(F).$$
Since $h^0(F)=4$, this completes the proof that \eqref{final} is not surjective.

(ii) The stable bundles $E$ constructed in (i) have $\gamma(E)=\frac13(20-6)=\frac{14}3$. On the other hand, since $\Clt\ge\frac92$, we have $\Clth\ge\frac{10}3$ by \cite[Theorem 4.1]{cl1}. Checking the proof of \cite[Theorem 4.1]{cl1}, one sees easily that one can replace $\frac{2\Clt+1}3$ in the statement of the theorem by $\min\left\{\frac{2\Cl+1}3,\frac{2\Clt+2}3\right\}$. So, in  our case, $\Clth\ge\frac{11}3$.

(iii) We have $\beta(3,20,6)=9g-8-6(6-20+3g-3)=-5$. The bundle $E$ is semistable and therefore stable since $\gcd(3,20)=1$, so $B(3,20,6)\ne\emptyset$. 
\end{proof}

\section{Questions and comments}\label{ques}
In this section, we raise a number of interesting questions with some observations on possible answers.  
\begin{ques}\label{ques1}\begin{em}
Can one find further examples of semistable bundles $E$ of rank $3$ with $h^0(E)=6$ and $\gamma(E)<\Cl$? 
\end{em}\end{ques}
In attempting to answer this, we note first
\begin{prop}\label{comm1}
Let  $E$ be a semistable bundle of rank $3$ on $C$ with $h^0(E)=6$. Then
\begin{equation}\label{h0=6}
\gamma(E)\ge\min\left\{\frac{d_9}3-2, d_1-2,\frac{2d_2}3-2,\frac{d_4}2-2\right\}.
\end{equation}
\end{prop}
\begin{proof}
If $E$ has no proper subbundle $F$ with $h^0(F)>r_F$, then $d_E\ge d_9$ by Lemma \ref{lempr}. So $\gamma_E\ge\frac13(d_9-6)=\frac{d_9}3-2$.

If $E$ has a line subbundle $L$ with $h^0(L)\ge2$, then $d_L\ge d_1$, so by semistability $d_E\ge3d_1$, giving $\gamma(E)\ge\frac13(3d_1-6)=d_1-2$.

If $E$ has a subbundle $F$ of rank $2$ with $h^0(F)=3$ and no line subbundle with $h^0\ge2$, then $d_F\ge d_2$. Moreover, $h^0(E/F)\ge3$, so $d_{E/F}\ge d_2$. Thus $d_E\ge2d_2$ and $\gamma(E)\ge\frac{2d_2}3-2$.

Finally, if $E$ has a subbundle $F$ of rank $2$ with $h^0(F)\ge4$ but no line subbundle with $h^0\ge2$, then, by Lemma \ref{lempr} again, we have $d_F\ge d_4$. So $d_E\ge\frac{3d_4}2$ and $\gamma(E)\ge\frac{d_4}2-2$.
\end{proof}

\noindent{\bf Comment 1.}
The value $d_1-2$ in \eqref{h0=6} can always be attained (by a direct sum of $3$ line bundles of degree $d_1$ with $h^0=2$), but it is not clear whether it can be attained by a stable bundle; in any case $d_1-2\ge\Cl$, so this is not interesting from the point of view of Question \ref{ques1}. In general, the construction of \eqref{e1l2} seems the most likely to yield examples and the main obstacle is that we need to prove that $h^0(E_2\otimes E_1)>9$.  For Petri curves of low genus, we have
\begin{itemize}
\item $g=7$: \eqref{h0=6} gives $\gamma(E)\ge\frac83$. This could be attained only by a bundle of the form \eqref{e1l2} and such bundles exist if and only if $h^0(E_2\otimes E_1)>9$. Apart from this, $\gamma(E)\ge\Cl$.
\item $g=8,10$: \eqref{h0=6} gives $\gamma(E)\ge d_1-2=\Cl$, so there is nothing to prove.
\item $g=9$: \eqref{h0=6} gives $\gamma(E)\ge\frac{10}3$ and this value can be attained by Theorem \ref{g=9}.
\item $g=11$: \eqref{h0=6} gives $\gamma(E)\ge\frac{13}3$. By Theorem \ref{g=11}, the value $\frac{14}3$ can be attained, but we do not know about $\frac{13}3$. Any bundle attaining this value would have to be of degree $19$ and to have no proper subbundle $F$ with $h^0(F)>r_F$. 
\item $g=12$: \eqref{h0=6} gives $\gamma(E)\ge\frac{14}3$. This value could be attained by a bundle of the form \eqref{e1l2} or by a bundle possessing no subbundle $F$ with $h^0(F)>r_F$; we do not know whether any such bundles exist.
\end{itemize}

\noindent{\bf Comment 2.}
If we restrict attention to stable bundles on Petri curves, Question \ref{ques1} can be rephrased in terms of Brill-Noether loci. In fact, on a Petri curve, we have (using \eqref{dr})
\begin{itemize}
\item $g$ odd: $\beta(3,d,6)<0$ if and only if $d\le 3d_1-1$;
\item $g$ even: $\beta(3,d,6)<0$ if and only if $d\le 3d_1+1$. In this case the Brill-Noether locus $B(3,3d_1+1,6)$ is non-empty. (Take 3 line bundles $L_1$, $L_2$, $L_3$ of degree $d_1$ with $h^0(L_i)=2$ and no two of the $L_i$ isomorphic, and take $E$ to be a general positive elementary transformation $0\to L_1\oplus L_2\oplus L_3\to E\to \tau\to0$ with $\tau$ of length $1$; then $E\in B(3,3d_1+1,6)$.)
\end{itemize}
We have also $\Cl=d_1-2$; so, for $E$ of rank $3$ with $h^0(E)=6$, the condition $d_E\le3d_1\pm1$ is equivalent to $\gamma(E)\le\Cl\pm\frac13$. We therefore have the following version of Question \ref{ques1}.
\begin{ques}\label{ques2}\begin{em} Do there exist non-empty Brill-Noether loci $B(3,d,6)$ on a Petri curve $C$ with negative expected dimension in addition to those listed in Comment 2 above and those of Theorems \ref{g=9} and \ref{g=11}?
\end{em}\end{ques}

\noindent{\bf Comment.} 
The case $d_E=3d_1$, where $\gamma(E)=\Cl$, is interesting here. There certainly exist semistable bundles in this case, but it is not clear whether stable bundles exist.

\begin{ques}\label{ques3}\begin{em}
Can one calculate $\Clth$ precisely or at least obtain a better estimate than that of \cite{cl1}?  
 \end{em}\end{ques}

\noindent{\bf Comment.}
For a Petri curve of genus $11$, we have slightly improved the estimate of \cite{cl1} in the proof of Theorem \ref{g=11}(ii) and this improvement applies to other curves. The arguments of \cite{cl1} suggest that one example to be considered for a low value of $\gamma(E)$ would be a bundle of degree $2g+3$ expressible in the form
$$0\to F\to E\to L\to0,$$
where $F$ is a semistable bundle of rank 2 and degree $d_2$ with $h^0(F)=3$, $L$ is a line bundle with the maximal number of sections possible for its degree and $h^0(E)=3+h^0(L)$. Possible examples are
\begin{itemize}
\item $C$ Petri of genus 7: $d_L=10=d_4$, so $h^0(E)=8$ and $\gamma(E)=\frac73$; if such a bundle exists, it computes $\Clth$;
\item $C$ Petri of genus 9: $d_L=13=d_5$, so $h^0(E)=9$ and $\gamma(E)=3$; again, if such a bundle exists, it computes $\Clth$;
\item $C$ Petri of genus 11: $d_L=15, h^0(L)=6$, so $h^0(E)=9$ and $\gamma(E)=\frac{13}3$; even if such a bundle exists, it may not compute $\Clth$.
\end{itemize} 

\begin{ques}\begin{em}
On a Petri curve of genus $11$, does there exist a stable bundle of rank $2$ and degree $13$ with $h^0=4$? 
\end{em}\end{ques}

\noindent{\bf Comment.}
As we have seen in Proposition \ref{g11<}, this is the only way in which one could have $\Clt<\Cl$ on such a curve. By \cite[Theorem 3.2 and Remark 5.4]{gmn}, such a bundle exists if and only if there is a non-degenerate morphism $C\to{\mathbb P}^4$ of degree $13$ whose image is contained in a quadric.
\begin{ques}\begin{em}
What about $n=4$? 
\end{em}\end{ques}

\noindent{\bf Comment.}
There are now two obstacles to using the method of Section \ref{mainsec}; Propositions \ref{ss} and \ref{mukai} both use  $n=3$. It may be possible to generalise the first of these propositions; the second looks more problematic. Another possible method is to use extensions $0\to E\to G\to M\to0$ with $E$ of rank $3$ and degree $2d_2$ with $h^0(E)=6$ and $M$ a line bundle of degree $d_1$ with $h^0(M)=2$; one still needs to prove semistability and show that the multiplication map
$$H^0(M)\otimes H^0(K_C\otimes E^*)\to H^0(M\otimes K_C\otimes E^*)$$
 is not surjective.
\begin{ques}\begin{em}
Can one find examples with $\Clth>\Clt$ (or more generally $\Clnn>\Cln$)? 
\end{em}\end{ques}

\noindent{\bf Comment.}
Any example would show that the hypotheses of \cite[Theorem 2.4]{ln} can fail. Genus $11$ is the first case where this might happen, but note that the bundles $E$ constructed in the proof of Theorem \ref{g=11} are all generated, so the conclusion of \cite[Theorem 2.4]{ln} could still hold.

\section{Postscript}\label{ps}

Since this paper was completed in September 2010, there have been remarkable developments in the construction of bundles providing counter-examples to Mercat's conjecture and relating them to Koszul cohomology, the maximal rank conjecture and the geometry of the moduli space of curves \cite{fo,ln2,ln3,fo2}. It is interesting (and probably significant) to note that all currently known counter-examples involve curves lying on K3 surfaces. Most of this work concerns bundles of rank $2$ and in particular \cite[Theorem 1.3]{fo2} gives a negative answer to Question 5.5 for a general curve of genus $11$. The paper \cite{fo2} also contains a significant result for bundles of rank $3$, showing that $\Clth<\Cl$ for curves of genus $g\ge11$ of maximal Clifford index which lie on K3 surfaces, thus extending the results of this paper and providing an answer to Question 5.1 \cite[Corollary 1.6]{fo2}.


\begin{thebibliography}{CAV}
\bibitem{bgn} S. B. Bradlow, O. Garc\'{\i}a-Prada, V. Mu\~noz and P. E. Newstead: 
\emph{Coherent systems and Brill-Noether theory}. 
Internat. J. Math. {\bf 14} (2003), 683--733.
\bibitem{fo} G. Farkas and A. Ortega:
\emph{The maximal rank conjecture and rank two Brill-Noether theory}.
arXiv:1010.4060. To appear in Pure and Appl. Math. Quarterly 7, no. 4 (2011).
\bibitem{fo2} G. Farkas and A. Ortega:
\emph{Higher rank Brill-Noether theory on sections of K3 surfaces}.
arXiv:1102.0276.
\bibitem{gmn} I. Grzegorczyk, V. Mercat and P. E. Newstead,
\emph{Stable bundles of rank 2 with 4 sections}.
arXiv:1006.1258v2, to appear in Internat. J. Math.
\bibitem{gt} I. Grzegorcyk and M. Teixidor i Bigas:
\emph{Brill-Noether theory for stable vector bundles}.
In: Moduli Spaces and Vector Bundles, 29--50, London Math. Soc. Lecture Note Ser., {\bf 359}, Cambridge University Press, Cambridge, 2009.
\bibitem{cl} H. Lange and P. E. Newstead: 
\emph{Clifford indices for vector bundles on curves}.
In:  Affine Flag Manifolds and Principal Bundles, 165--202, Trends in Mathematics, Birkh\"auser, Basel, 2010.
\bibitem{ln} H. Lange and P. E. Newstead: 
\emph{Generation of vector bundles computing Clifford indices}.
Arch. Math. {\bf94} (2010), 529--537.
\bibitem{cl1} H. Lange and P. E. Newstead: 
\emph{Lower bounds for Clifford indices in rank three}.
Math. Proc. Camb. Philos. Soc. {\bf 150} (2011), 23-33, doi:10.1017/S0305004110000502.
\bibitem{ln2} H. Lange and P. E. Newstead:
\emph{Further examples of stable bundles of rank 2 with 4 sections}.
arXiv:1011.4089. To appear in Pure and Appl. Math. Quarterly 7, no. 4 (2011).
\bibitem{ln3} H. Lange and P. E. Newstead:
\emph{Vector bundles of rank 2 computing Clifford indices}.
arXiv:1012.0469.
\bibitem{Laz} R.~Lazarsfeld,
\emph{Brill-Noether-Petri without degeneration}.
J. Differential Geom. {\bf 23} (1986), 299--307.
\bibitem{m} V. Mercat:
\emph{Clifford's theorem and higher rank vector bundles}.
Internat. J. Math. {\bf13} (2002), 785--796.
\bibitem{mm} D. Mumford:
\emph{Theta characteristics of an algebraic curve}.
Ann. scient. \'Ec. Norm. Sup., $4^e$ s\'erie, t. 4 (1971), 181--192.
\bibitem{mu1} S. Mukai:
\emph{Curves and symmetric spaces}.
Proc. Japan Acad. {\bf68}, ser. A (1992), 7--10.
\bibitem{mu2} S. Mukai:
\emph{Curves and symmetric spaces, II}.
RIMS-1395 (2003), to appear in Ann. of Math.
\bibitem{mu3} S. Mukai:
\emph{Addendum to ``Curves and symmetric spaces, II''}.
Addendum RIMS-1395 (2010).
\bibitem{pr} K. Paranjape and S. Ramanan: 
\emph{On the canonical ring of a curve}.
Algebraic Geometry and Commutative Algebra VoL II, 503--516, Kinokuniya, Tokyo, 1988. 

\end{thebibliography}
\end{document}